\documentclass[english,a4paper,11pt,draft]{article}
\usepackage{amsmath,amssymb,amsfonts,latexsym, amsthm, mathrsfs}
\usepackage[active]{srcltx}
\catcode`\@=11
\@addtoreset{equation}{section}

\catcode`\@=12

\newtheorem{Theorem}{Theorem}[section]
\newtheorem{Lemma}[Theorem]{Lemma}
\newtheorem{Proposition}[Theorem]{Proposition}

\newcommand{\Rn}{{\mathbb R}^{N}}
\newcommand{\N}{\mathbb{N}}

\newcommand{\al} {\alpha}

\newcommand{\la} {\lambda}

\newcommand{\pa}{\partial}
\newcommand{\na} {\nabla}

\newcommand{\no}{\nonumber}
\newcommand{\bdw}{\partial\Omega}

\newcommand{\R}{\mathbb{R}}

\newcommand{\deb}{\rightharpoonup}

\newcommand{\Iom}{\displaystyle\int_{\Omega}}

\begin{document}

\title{Infinitely many sign changing solutions of an elliptic problem involving critical Sobolev and Hardy-Sobolev exponent}

\author{Mousomi Bhakta
\footnote{Dept of Mathematics, Indian Institute of Science Education and Research, Pashan, Pune-411008, India. E-mail:  {\it mousomi@iiserpune.ac.in}}}

\date{}

\maketitle

\begin{abstract}
\noindent
\footnotesize
We study the existence and multiplicity of sign changing solutions of the following equation 
\begin{equation*}
\begin{cases}
-\Delta u = \mu |u|^{2^{\star}-2}u+\frac{|u|^{2^{*}(t)-2}u}{|x|^t}+a(x)u \quad\text{in}\quad \Omega, \\
 u=0 \quad\text{on}\quad\bdw,
\end{cases}
\end{equation*} where $\Omega$ is a bounded domain in $\Rn$, $0\in\bdw$, all the principal curvatures of $\bdw$ at $0$ are negative and $\mu\geq 0, \ \ a>0, \ \ N\geq 7, \ \ 0<t<2$, \ \ $2^{\star}=\frac{2N}{N-2}$ and $2^{\star}(t)=\frac{2(N-t)}{N-2}$. 

\bigskip

\medskip

\end{abstract}

\section{Introduction}\label{intro} In this article we study the following elliptic problem:
\begin{equation}\label{eq:prob}
\begin{cases}
 - \Delta u = \mu|u|^{2^{\star}-2}u+\frac{|u|^{2^{\star}(t)-2}u}{|x|^t}+a(x)u \quad\text{in}\quad \Omega,\\
u = 0  \quad\text{on}\quad\bdw, \\
 \end{cases}\end{equation}
where $\Omega$ is a bounded domain in $\Rn$, $\bdw$ is $C^3$, $0\in\bar\Omega$ and
\begin{equation}\label{assump}
\mu\geq 0,\ \ a\in C^1(\bar\Omega), \ \ a>0, \ \  0<t<2, \quad \text{and} \ \ 2^{\star}(t)=\frac{2(N-t)}{N-2}.
\end{equation}
Here $2^{\star}$ is the usual critical sobolev exponent $\frac{2N}{N-2}$. 

By a solution of the above equation we mean $u \in  H_0^1 (\Omega)$ satisfying
$$\int\limits_{\Omega}\nabla u \cdot \nabla v dx = \mu\int\limits_{\Omega}|u|^{2^{\star}-2}uv dx+\frac{|u|^{2^{\star}(t)-2}uv}{|x|^t} dx+\int_{\Omega}a(x)uv dx \quad \forall v \in  H_0^1 (\Omega). $$

Equivalently, $u$ is a critical point of the functional $I$,
\begin{equation}
 I(u) = \frac12 \int\limits_{\Omega}|\nabla u|^2\; dx-  \frac{1}{2} \int\limits_{\Omega}a(x)|u|^2 \; dx - \frac{\mu}{2^{*}}\Iom|u|^{2^{*}} \; dx-\frac{1}{2^{*}(t)} \int\limits_{\Omega}\frac{|u|^{2^{*}(t)}}{|x|^t}\;dx.
\end{equation}
$I$ is a well defined $C^1 $ functional on $H_0^1 (\Omega)$ for any open subset of $\R^N$, thanks to the following Hardy-Sobolev Inequality ,

\textbf {Hardy-sobolev inequality}: Let $N\geq 3$, $0\leq t<2$. Then there exist a positive constant $C=C(N,t)$ such that 
\begin{equation}\label{Sobolev}
\left(\int\limits_{\R^N} \frac {|u|^{2^{\star}(t)}} {|x|^t} dx \right)^{\frac{2}{2^{*}(t)}} \leq   
C \int\limits_{\R^N}|\nabla u|^2  dx \quad\forall\quad u\in C^{\infty}_0(\Rn). 
\end{equation}

Equation \eqref{eq:prob} involves multiple critical exponents, namely, critical Sobolev exponent and Hardy-Sobolev exponent. In recent years a lot of attention has been given to the existence of nontrivial solutions of problem \eqref{eq:prob}. As it is mentioned in \cite{YY}, one can apply the pioneering idea of Brezis and Nirenberg \cite{BN}, to obtain a positive solution of \eqref{eq:prob}. 

\vspace{2mm}

When $\Omega=\Rn$ and the function $a$ is singular at the origin, existence of positive solution to  more general type equations involving multiple critical exponents was studied by Fillippucci, et all \cite{FPR} using Mountain Pass Lemma of Ambrosetti and Rabinowitz \cite{AR}. 
There the crucial step is to show that the mountain pass value is strictly less than the first energy level at which the Palais-Smale condition fails. For the existence of the mountain pass solution of \eqref{eq:prob}, we also refer \cite{LL} and the references there-in. As it is pointed out in \cite{YY} that, when $0\in\bdw$, the mean curvature of $\bdw$ at $0$ plays an important role in the existence of mountain pass solutions (see \cite{CL}, \cite{GK}, \cite{GR}, \cite{LL}).

\vspace{2mm}

In \cite{YY}, Yan and Yang have considered the problem \eqref{eq:prob} with $a\in C^1(\bar\Omega)$ and $a(0)>0$. They have proved the existence of infinitely many solutions using the compactness of the solutions of Brezis-Nirenberg type problem established by  Devillanova and Solimini \cite{DS} for $N\geq 7$. But \cite{YY} does not have any information about the existence and multiplicity of sign changing solutions. Also, it is wroth mentioning that, one can not obtain the existence of infinitely many sign changing solutions of \eqref{eq:prob} by adopting the method of \cite{YY}. Therefore
a natural question which is still open
is whether \eqref{eq:prob} has infinitely many sign changing solutions for any $a\in C^1(\bar\Omega)$ such that $a>0$.

\vspace{2mm}

A very important result by Schechter and Zou \cite{SZ} asserts that there exists infinitely many sign changing solutions to the Brezis-Nirenberg problem in higher dimension. Very recently this kind of technique were also used in \cite{G} and \cite {GS} to prove the existence of infinitely many sign changing solutions of Hardy-Sobolev-Maz'ya type equations and Brezis-Nirenberg problem in hyperbolic space respectively. 

\vspace{2mm}

So far in the literature the only two papers that deal with the sign changing solution of \eqref{eq:prob} are \cite{CP} and \cite {CZ}. In \cite{CP}, a pair of sign changing solutions and in \cite{CZ} infinitely many sign changing solutions were obtained for \eqref{eq:prob} when $t=2$, the function $a$ is constant and $0\in\Omega$. More precisely in \cite{CZ}, the following equation problem was studied:
\begin{equation}\label{eq:cz-prob}
\begin{cases}
 - \Delta u = |u|^{2^{\star}-2}u+\mu\frac{u}{|x|^2}+\la u \quad\text{in}\quad \Omega,\\
u = 0  \quad\text{on}\quad\bdw, \\
 \end{cases}\end{equation} 
where $\la>0$ is a constant, $0\in\Omega$. We like to point out as in \cite{YY} that, there is some differences between the case $t=2$ and $t\in(0,2)$. If $t=2$, solutions of \eqref{eq:cz-prob} are singular at $0$ and that was the main reason to impose the condition $\mu\in\displaystyle\left(0, \frac{(N-2)^2}{4}-4\right)$ in \cite{CZ}. If $t\in(0,2)$, no such condition is needed. Also, there  are some differences between the cases when $a$ is a constant function and not a constant function.   

\vspace{2mm}

Our main theorem is the following:
\begin{Theorem}\label{t:main}
Let $N\geq 7$, $0\in\bdw$, all the principal curvatures of $\bdw$ at $0$ be negative and \eqref{assump} hold. Then  Equation \eqref{eq:prob} has infinitely many sign changing solutions.
\end{Theorem}

We will prove this theorem by applying an abstract theorem by Schechter and Zou \cite[theorem 2]{SZ}. However we can not apply that theorem directly as $I$, that is, the variational problem corresponding to \eqref{eq:prob} does not satisfy Palais-Smale condition. To overcome this difficulty we consider the perturbed subcritical problem 
\begin{equation}\label{eq:sub-prob}
\begin{cases}
 - \Delta u = \mu|u|^{2^{\star}-2-\epsilon_n}u+\frac{|u|^{2^{\star}(t)-2-\epsilon_n}u}{|x|^t}+a(x)u \quad\text{in}\quad \Omega,\\
u = 0  \quad\text{on}\quad\bdw, \\
\end{cases}\end{equation}
where $0<\epsilon_n\downarrow 0.$ We will prove that for each $\epsilon_n$, \eqref{eq:sub-prob} has a sequence of sign changing solution $\{u_{n,l}\}_{l=1}^{\infty}$ and we will show that Morse index of $\{u_{n,l}\}_{l=1}^{\infty}$ has a lower bound. Then we will prove that for fixed $l$, $\sup_{n\in\N}||u_{n,l}||_{H^1_0(\Omega)}<\infty$.  

\vspace{2mm}

We organize the paper as follows. In Section 2, we prove the existence and the estimate of Morse index of sign changing solution of \eqref{eq:sub-prob}. Using this, in Section 3 we 
prove Theorem \ref{t:main}. In Section 4, we prove a nonexistence result in star shaped domain under some condition on the function $a$.

{\bf Notation}: Through out this paper we denote the norm in $H^1_0(\Omega)$ by $||u||=\displaystyle\left(\Iom|\na u|^2 dx\right)^\frac{1}{2}$ and $|u|_{q,t,\Omega}:=\displaystyle\left(\Iom\frac{|u|^q}{|x|^t}dx\right)^\frac{1}{q}$. We say $u\in L^q_t(\Omega)$ If $|u|_{q,t,\Omega}<\infty$.  

\section{Existence of sign changing critical points}
Consider the weighted eigenvalue prob:
\begin{equation}\label{eq:e-v}
-\Delta u=\la a(x)u \quad\text{in}\quad\Omega; \quad u=0 \quad\text{on}\quad\bdw.
\end{equation}
Since $a\in C^1(\bar\Omega)$ and strictly positive, the above operator has infinitely many eigenvalues $\{\la_i\}_{i=1}^{\infty}$ such that $0<\la_1(\Omega)<\la_2(\Omega)\leq\la_3(\Omega)\leq\cdots\leq\la_l(\Omega)\leq\cdots$. Therefore, we can write 
\begin{equation}\label{eq:la}
\la_1=\inf_{u\in H^1_0(\Omega), u\not=0}\frac{\Iom|\na u|^2 dx}{\Iom a(x)u^2 dx}.
\end{equation} 
Let $\phi_i$ be the orthonormal eigen vectors corresponding to $\la_i$ where we know $\phi_1>0$. We denote $E_k=\text{span}\{\phi_1,\cdots,\phi_k\}$. Then $E_k\subset E_{k+1}$ and $H^1_0(\Omega)=\overline{\cup_{k=1}^{\infty} E_k}$ (see \cite{GT}).

\begin{Lemma}Suppose all the assumptions in Theorem \ref{t:main} hold. In addition, if $\la_1\leq 1$, then Equation \eqref{eq:prob} has infinitely many sign changing solutions.
\end{Lemma}
\begin{proof}
By multiplying the Equation \eqref{eq:prob} by $\phi_1$ and integrating by parts, it is easy to check that if $\la_1\leq 1$, then any nontrivial solution of \eqref{eq:prob} has to change sign. Also by \cite[Theorem 1.2]{YY}, it follows that Equation \eqref{eq:prob} has infinitely many solutions. Therefore the lemma follows. 
\hfill$\square$ 
\end{proof}

{\bf Therefore now onwards we assume $\la_1>1$}.  We fix $\epsilon_0>0$ small enough and choose a sequence $\epsilon_n\in(0,\epsilon_0)$
 such that $\epsilon_n\downarrow 0$ in \eqref{eq:sub-prob}.

We define the energy functional corresponding to \eqref{eq:sub-prob} as 
\begin{equation}
 I_n(u) = \frac12 \int\limits_{\Omega}|\nabla u|^2\; dx-  \frac{1}{2} \int\limits_{\Omega}a(x)|u|^2 \; dx - \frac{\mu}{2^{*}-\epsilon_n}\Iom|u|^{2^{*}-\epsilon_n} \; dx-\frac{1}{2^{*}(t)-\epsilon_n} \int\limits_{\Omega}\frac{|u|^{2^{*}(t)-\epsilon_n}}{|x|^t}\;dx.
\end{equation}
Then $I_n$ is an even $C^2$ functional on $H^1_0(\Omega)$. Also, $I_n$ satisfies Palais-Smale condition for each $n$. Therefore by Ambrosetti and Rabinowitz \cite{AR}, \eqref{eq:sub-prob} has infinitely many critical points $\{u_{n,l}\}_{l=1}^{\infty}$. More precisely, it follows from \cite{R} that there are positive numbers $c_{n,l}$, $l=1,2, \cdots,$ with $c_{n,l}\uparrow\infty$ as $\l\uparrow\infty$ and $I_n(u_{n,l})=c_{n,l}$. 
We define the augmented Morse index of $u_{n,l}$ by $m^{*}(u_{n,l})$ as follows:
$$m^{*}(u_{n,l}):=\text{max}\{\text{dim}\  H: H\subset H^1_0(\Omega)\ \ \text{is a subspace such that}\ \ I_n^{''}(v,v)\leq 0 \quad\forall v\in H^1_0(\Omega)\}.$$

For each $\epsilon_n$, We define
$$||u||_{*,n}=\mu|u|_{L^{2^{*}-\epsilon_n}(\Omega)}+\displaystyle\left(\Iom\frac{|u|^{2^{*}(t)-\epsilon_n}}{|x|^t}\right)^\frac{1}{2^{*}(t)-\epsilon_n}; \quad\forall\quad u\in H^1_0(\Omega).$$

Here we state two lemmas in the same spirit as in \cite{G}. therefore we omit the proof.
\begin{Lemma}\label{G1}If $\Omega$ is a bounded domain in $\Rn$, $0\leq t<2$, $N\geq 3$ and $1\leq q\leq p<\infty$, then $L^p_t(\Omega)\subset L^q_t(\Omega)$ and the inclusion is continuous.
\end{Lemma}
\begin{Lemma}\label{G2}
Let $1\leq q<2^{*}(t)$, $0\leq t<2$ and $N\geq 3$. Then the embedding $H^1_0(\Omega)\subset L^q_t(\Omega)$ is compact.
\end{Lemma}

Therefore by \eqref{Sobolev}, Lemma \ref{G1} and Lemma \ref{G2}, we have $||u||_{*,n}\leq C||u||_{H^1_0(\Omega)}$ where $C$ is independent of $n$ and $||u_l-u||_{*,n}\to 0$ whenever $u_l\deb u$ in $H^1_0(\Omega)$. Hence $\bf{(A_0)}$ of \cite{SZ} is satisfied.

We define, $\mathcal{P}:=\{u\in H^1_0(\Omega): u\geq 0\}$ and $\mathcal{K}_n:=\{u\in H^1_0(\Omega): I_n^{'}(u)=0\}$. For each $\delta>0$, we define $\mathcal{D}(\delta):=\{u\in H^1_0(\Omega): \text{dist}(u,\mathcal{P})<\delta\}$ .

The gradient $I_n^{'}$ is of the form $I_n^{'}(u)=u-K_n(u)$, where $K_n: H^1_0(\Omega)\to H^1_0(\Omega)$ is a continuous operator. In the next proposition, we will see how the operator, $K_n$, behaves on $\mathcal{D}(\delta)$. 

\begin{Proposition}\label{p:1}
Let $\la_1>1$ and \eqref{assump} hold. Then for any $\delta_0>0$ small enough, \\
$K_n(\pm\mathcal{D}(\delta_0))\subset\pm\mathcal{D}(\delta)\subset\pm\mathcal{D}(\delta_0)$ for some $\delta\in(0,\delta_0)$. Moreover, $\pm\mathcal{D}(\delta_0)\cap\mathcal{K}_n\subset\mathcal{P}.$
\end{Proposition}

\begin{proof}
First we note that $K_n(u)$ can be decomposed as 
$K_n(u)=L(u)+G_n(u)$, where $L(u), G_n(u)\in H^1_0(\Omega)$ are the unique solutions of the following equations:
$$-\Delta(L(u))=a(x)u; \quad -\Delta (G_n(u))=\mu|u|^{2^{*}-2-\epsilon_n}u+\frac{|u|^{2^{*}(t)-2-\epsilon_n}u}{|x|^t}.$$
In other words, $L(u)$ and $G_n(u)$ are uniquely determined by 
\begin{equation}\label{eq:L_n}
\displaystyle\left<L(u), v\right>_{H^1_0(\Omega)}=\Iom a(x)uv dx,
\end{equation}
\begin{equation}\label{eq:G_n}
\left<G_n(u), v\right>_{H^1_0(\Omega)}=\mu\Iom|u|^{2^{*}-2-\epsilon_n}uv dx+\Iom\frac{|u|^{2^{*}(t)-2-\epsilon_n}uv}{|x|^t}dx.
\end{equation}

We claim that, if $u\in\mathcal{P}$ then $L(u), G_n(u)\in\mathcal{P}$. To see this, let $u\in\mathcal{P}$. Then
$$-\displaystyle\Iom|\na L(u)^{-}|^2dx=\left<L(u), L(u)^{-}\right>_{H^1_0(\Omega)}=\Iom a(x)uL(u)^{-}\geq 0,$$
which immediately implies $L(u)\in\mathcal{P}$. Similary, we have $G_n(u)\in\mathcal{P}$.

Using \eqref{eq:L_n} we see,
\begin{equation}
||L(u)||^2=\displaystyle\left<L(u), L(u)\right>_{H^1_0(\Omega)}=\Iom a(x)uL(u)\leq\left(\Iom a(x)u^2 dx\right)^\frac{1}{2}\left(\Iom a(x)L(u)^2 dx\right)^\frac{1}{2}
\end{equation}
Therefore using \eqref{eq:la} in the above expression, we obtain
$$||L(u)||^2_{H^1_0(\Omega)}\leq\frac{1}{\la_1}||u||_{H^1_0(\Omega)}||L(u)||_{H^1_0(\Omega)},$$ which yields $||L(u)||_{H^1_0(\Omega)}\leq\frac{1}{\la_1}||u||_{H^1_0(\Omega)}.$

For any $u\in H^1_0(\Omega)$ we consider $v\in \mathcal{P}$ such that $\text{dist}(u,\mathcal{P})=||u-v||$. Then
\begin{equation}\label{eq:Lu}
\text{dist}(L(u),\mathcal{P})\leq||L(u)-L(v)||\leq\frac{1}{\la_1}||u-v||\leq\frac{1}{\la_1}\text{dist}(u,\mathcal{P}).
\end{equation}
Next, 
\begin{eqnarray}\label{1}
\text{dist}(G_n(u),\mathcal{P})||G_n(u)^{-}|| &\leq&||G_n(u)-G_n(u)^{+}|| ||G_n(u)^{-}||=||G_n(u)^{-}||^2\no\\
&\leq&-\displaystyle\left<G_n(u), G_n(u)^{-}\right>_{H^1_0(\Omega)}\no\\
&=&-\mu\Iom|u|^{2^{*}-2-\epsilon_n}u G_n(u)^{-}-\Iom\frac{|u|^{2^{*}(t)-2-\epsilon_n}u G_n(u)^{-}}{|x|^t}\no\\
&\leq&\mu\Iom|u|^{2^{*}-2-\epsilon_n}u^{-} G_n(u)^{-}+\Iom\frac{|u|^{2^{*}(t)-2-\epsilon_n}u^{-} G_n(u)^{-}}{|x|^t}\no\\
&=&\mu\Iom|u^{-}|^{2^{*}-1-\epsilon_n} G_n(u)^{-}+\Iom\frac{|u^{-}|^{2^{*}(t)-1-\epsilon_n} G_n(u)^{-}}{|x|^t}\no\\
&\leq&\mu\left(\Iom|u^{-}|^{2^{*}-\epsilon_n}dx\right)^\frac{2^{*}-1-\epsilon_n}{2^{*}-\epsilon_n}\left(\Iom|G_n(u)^{-}|^{2^{*}-\epsilon_n}dx\right)^\frac{1}{2^{*}-\epsilon_n}
\no\\
&+&\left(\Iom\frac{|u^{-}|^{2^{*}(t)-\epsilon_n}}{|x|^t}dx\right)^\frac{2^{*}(t)-1-\epsilon_n}{2^{*}(t)-\epsilon_n}\left(\Iom\frac{|G_n(u)^{-}|^{2^{*}(t)-\epsilon_n}}{|x|^t}dx\right)^\frac{1}{2^{*}(t)-\epsilon_n}.
\end{eqnarray}
By using Lemma \ref{G1} and the Sobolev inequality, the last term in the RHS of the above expression \eqref{1} can be shown less than $$C\displaystyle\left(|u^{-}|_{L^{2^{*}-\epsilon_n}}^{2^{*}-1-\epsilon_n}+|u^{-}|_{L^{2^{*}(t)-\epsilon_n}_t}^{2^{*}(t)-1-\epsilon_n}\right)||G_n(u)^{-}||_{H^1_0(\Omega)}.$$ Therefore we obtain from \eqref{1},
\begin{equation*}
\text{dist}(G_n(u), \mathcal{P})\leq C\displaystyle\left(|u^{-}|_{L^{2^{*}-\epsilon_n}}^{2^{*}-1-\epsilon_n}+|u^{-}|_{L^{2^{*}(t)-\epsilon_n}_t}^{2^{*}(t)-1-\epsilon_n}\right)
\end{equation*}
Using \eqref{Sobolev}, it is easy to check that, from the above equation we can obtain  
$$\text{dist}(G_n(u), \mathcal{P})\leq C\displaystyle\left(\text{dist}(u,\mathcal{P})^{2^{*}-1-\epsilon_n}+\text{dist}(u,\mathcal{P})^{2^{*}(t)-1-\epsilon_n}\right) \quad \forall\quad u\in H^1_0(\Omega),$$(see \cite{CZ}, \cite{G}). As $\la_1>1$, we choose $\nu\in (\frac{1}{\la_1}, 1)$. Then there exists  $\delta_0>0$ such that if $\delta\leq\delta_0$, we have
\begin{equation}\label{eq:Gu}
\text{dist}(G_n(u),\mathcal{P})\leq\displaystyle\left(\nu-\frac{1}{\la_1}\right)\text{dist}(u, \mathcal{P}) \quad\forall\quad u\in\mathcal{D}(\delta).
\end{equation}
Combining \eqref{eq:Lu} and \eqref{eq:Gu}, we obtain $$\text{dist}(K_n(u),\mathcal{P})\leq\text{dist}(L(u),\mathcal{P})+\text{dist}(G_n(u),\mathcal{P})\leq \nu \ \text{dist}(u, \mathcal{P}) \quad\forall\quad u\in\mathcal{D}(\delta).$$
Hence we get, $K_n(\mathcal{D}(\delta_0)\subset\mathcal{D}(\delta))\subset\mathcal{D}(\delta_0))$ for some $\delta\in(0,\delta_0)$. Also if, dist$(u,\mathcal{P})<\delta_0$ and $I_n^{'}(u)=0$ that is, $u=K_n(u)$, then we have, dist$(u,\mathcal{P})=\text{dist}(K_n(u),\mathcal{P})\leq\nu\text{dist}(u,\mathcal{P})$, which immediately implies $u\in\mathcal{P}$. Similarly we can prove $K_n(-\mathcal{D}(\delta_0))\subset -\mathcal{D}(\delta)\subset-\mathcal{D}(\delta_0)$ for some $\delta\in(0,\delta_0)$ and  $-\mathcal{D}(\delta_0)\cap\mathcal{K}_n\subset\mathcal{P}.$ This completes the proof. 
\hfill$\square$
\end{proof}

\begin{Lemma}\label{l:1}
Let $\la_1>1$ and \eqref{assump} hold. Then for each $k$, $\lim_{||u||\to\infty, u\in E_k}I_n(u)=-\infty$
\end{Lemma}
\begin{Lemma}\label{l:2}
Let $\la_1>1$ and \eqref{assump} hold. Then for any $\al_1, \al_2>0$, there exists an $\al_3$ depending on $\al_1$ and $\al_2$ such that $||u||\leq\al_3$ for all $u\in I_n^{\al_1}\cap\{u\in H^1_0(\Omega): ||u||_{*,n}\leq\al_2\}$, where $I_n^{\al_1}=\{u\in H^1_0(\Omega): I_n(u)\leq\al_1\}$.
\end{Lemma}
The above two lemmas are quite standard and can be proved by using the similar technique as in \cite{G}.

\begin{Theorem}\label{t:sub}
Let $\la_1>1$ and \eqref{assump} hold. Then for each $n$, Equation \eqref{eq:sub-prob} has infinitely many sign changing solutions $\{u_{n,l}\}_{l=1}^{\infty}$ such that for each $l$, the sequence $\{u_{n,l}\}$ is bounded in $H^1_0(\Omega)$ and the augmented Morse index of $u_{n,l}$ is greater than or equal to $l$. 
\end{Theorem}
\begin{proof}
By applying Proposition \ref{p:1}, Lemma \ref{l:1} and Lemma \ref{l:2}, we see that $I_n$ satisfies all the assumptions $(A_1)-(A_3)$in \cite[Theorem 2]{SZ}. Hence by \cite[Theorem 2]{SZ}, $I_n$ has a sign changing critical point $u_{n,l}$ at the level $c_{n,l}$, where $c_{n,l}\leq\sup_{E_{l+1}}I_n$ and $m^{*}(u_{n,l})\geq l$. 

{\bf Claim}: There exists  positive constant $T_1$, independent of $n$ and $l$ such that $$c_{n,l}\leq T_1\la_{l+1}^\frac{2^{*}(t)-\epsilon_0}{2(2^{*}(t)-\epsilon_0-2)}.$$ 

To see this, we note that since $2^{*}(t)-\epsilon_0>2$, then
\begin{equation}\label{2}
||u||^2\leq \la_{l+1}|u|^2_{L^2(\Omega)}\leq C\la_{l+1}|u|^2_{L^{2^{*}(t)-\epsilon_0}_t(\Omega)} \quad\forall\quad  u\in E_{l+1},
\end{equation}
where $C>0$ is a constant independent of $n,l$. Since for $0< t<2$, $2^{*}(t)-\epsilon_0<2^{*}(t)-\epsilon_n$, 
by H\"{o}lder inequality, it is easy to check that, there exists constants $D_1, D'_1>0$, independent of $n, l$ such that $|u|_{L^{2^{*}(t)-\epsilon_0}_t(\Omega)}<D_1|u|_{L^{2^{*}(t)-\epsilon_n}_t(\Omega)}+D'_1$. Therefore,
\begin{eqnarray*}
I_n(u)&\leq&\frac{1}{2}\Iom|\na u|^2 dx-\frac{1}{2^{*}(t)-\epsilon_n}\Iom\frac{|u|^{2^{*}(t)-\epsilon_n}}{|x|^t}dx\\
&\leq&\frac{1}{2}||u||^2-D_2\Iom\frac{|u|^{2^{*}(t)-\epsilon_0}}{|x|^t}dx+D_3,
\end{eqnarray*}
where $D_2, D_3>0$ are constants independent of $n,l$. Using \eqref{2} in the above expression, we have for all $u\in E_{l+1}$, 
\begin{equation*}
I_n(u)\leq\frac{1}{2}||u||^2-D_4\la_{l+1}^{-\frac{2^{*}(t)-\epsilon_0}{2}}||u||^{L^{2^{*}(t)-\epsilon_0}}+D_3\leq D_5\la_{l+1}^\frac{2^{*}(t)-\epsilon_0}{2(2^{*}(t)-\epsilon_0-2)}+D_3\leq T_1\la_{l+1}^\frac{2^{*}(t)-\epsilon_0}{2(2^{*}(t)-\epsilon_0-2)},
\end{equation*}
where $D_i (i=1,\cdots,5)$ and $T_1$ are positive constants independent of $n,l$. Since energy of any critical point is non-negative, we conclude, $I_n(u_{n,l})\in[0, T_1\la_{l+1}^\frac{2^{*}(t)-\epsilon_0}{2(2^{*}(t)-\epsilon_0-2)}]$. 
Also we see that,
\begin{eqnarray*}
I_n(u_{n,l})&=&I_n(u_{n,l})-\frac{1}{2^{*}(t)-\epsilon_n}I'_n(u_{n,l})(u_{n,l})\\
&=&\displaystyle\left(\frac{1}{2}-\frac{1}{2^{*}(t)-\epsilon_n}\right)\Iom(|\na u_{n,l}|^2-a(x)u_{n,l}^2)dx\\
&+&\mu\left(\frac{1}{2^{*}(t)-\epsilon_n}-\frac{1}{2^{*}-\epsilon_n}\right)\Iom |u_{n,l}|^{2^{*}-\epsilon_n}dx\\
&\geq&\displaystyle\left(\frac{1}{2}-\frac{1}{2^{*}(t)-\epsilon_0}\right)\Iom(|\na u_{n,l}|^2-a(x)u_{n,l}^2)dx\\
&\geq&\displaystyle\left(\frac{1}{2}-\frac{1}{2^{*}(t)-\epsilon_0}\right)\left(1-\frac{1}{\la_1}\right)||u_{n,l}||^2
\end{eqnarray*}
As $\la_1>1$, coefficient in the RHS is strictly positive. Hence $\{u_{n,l}\}_{n=1}^{\infty}$ is bounded in $H^1_0(\Omega)$ for each $l$, which completes the proof.
\hfill$\square$
\end{proof}

\section{Proof of Theorem \ref{t:main}}
We start this section by quoting a Theorem from Yan and Yang \cite[Theorem 1.1]{YY}

\begin{Theorem}\label{t:YY}
Let $a\in C^1(\bar\Omega)$, $a(0)>0$ and $0\in\bdw$, all the principal curvatures of $\bdw$ at $0$ are negative. If $N\geq 7$, $\mu\geq 0$, then for any $u_n$ which is a solution of \eqref{eq:sub-prob} with $\epsilon=\epsilon_n\to 0$, satisfying $||u_n||\leq C$, for some constant independent of $n$, then $u_n$ converges strongly in $H^1_0(\Omega)$.
\hfill$\square$
\end{Theorem}

{\bf Proof of Theorem \ref{t:main}}: Combining Theorem \ref{t:sub} and Theorem \ref{t:YY}, we obtain  $u_{n,l}\to u_l$ in $H^1_0(\Omega)$ as $n\to\infty$. Then $\{u_l\}_{l=1}^{\infty}$ is a sequence of solution to the Equation \eqref{eq:prob} with energy $c_l\in [0, T_1\la_{l+1}^\frac{2^{*}(t)-\epsilon_0}{2(2^{*}(t)-\epsilon_0-2)}]$. Next, we claim that $u_l$ is sign changing for each $l$. To see this, we note that as $I'_n(u_{n,l})=0$, we get
$$\Iom\left(|\na u_{n,l}^{\pm}|^2-a(x)|u_{n,l}^{\pm}|^2\right)dx=\mu\Iom|u_{n,l}^{\pm}|^{2^{*}-\epsilon_n}dx+\Iom\frac{|u_{n,l}^{\pm}|^{2^{*}(t)-\epsilon_n}}{|x|^t}dx$$
Therefore, using \eqref{eq:la} we have
$$\left(1-\frac{1}{\la_1}\right)||u_{n,l}^{\pm}||^2\leq\mu\Iom|u_{n,l}^{\pm}|^{2^{*}-\epsilon_n}dx+\Iom\frac{|u_{n,l}^{\pm}|^{2^{*}(t)-\epsilon_n}}{|x|^t}dx .$$
Since $\big(1-\frac{1}{\la_1}\big)<1$, by using  Lemma \ref{G1} and Sobolev inequality \eqref{Sobolev} in the above expression, we obtain $||u^{\pm}_{n,l}||\geq C>0$, where $C$ is independent of $n$. This in turn implies, $
||u^{\pm}_{l}||\geq C'>0$. Hence the claim follows.

To complete the proof, the last thing remains to show that infinitely many $u_l's$ are different. This is equivalent to prove $\lim_{l\to\infty}I(u_l)=\lim_{l\to\infty}c_l=\infty$. We prove this by method of contradiction. Suppose, $\lim_{l\to\infty}c_l\leq c<\infty$. For each $l$, we can find $n_l>l$ such that $|c_{n_l,l}-c_l|<\frac{1}{l}$. Therefore, $\lim_{l\to\infty}c_{n_l,l}=\lim_{l\to\infty}c_l<c<\infty$. Since $I'_n(u_{n_l,l})=0$, once again it proves that $\{u_{n_l,l}\}$ is bounded in $H^1_0(\Omega)$. Therefore by Theorem \ref{t:YY}, $\{u_{n_l,l}\}$ converges in $H^1_0(\Omega)$ and augmented Morse index of $\{u_{n_l,l}\}_{l=1}^{\infty}$ remains bounded, which is a contradiction to the fact that $m^{*}(u_{n_l,l})\geq l$. This completes the proof.
\hfill$\square$

\section{Nonexistence result}

\begin{Theorem}\label{t:nonex}
Suppose $N\geq 3$, $a\in C^1(\bar\Omega)$ and $\big(a(x)+\frac{1}{2} x\cdot\na a\big)\leq 0$ for every $x\in\Omega$. Then Equation \eqref{eq:prob} does not have any nontrivial solution in a domain which is star shaped domain with respect to the origin.
\end{Theorem}
\begin{proof}
We will prove this using the Pohozaev  identity in the spirit of \cite{BS}. For $\epsilon>0$ and $R>0$, define $\phi_{\epsilon, R}:=\phi_{\epsilon}(x)\psi_R(x)$, where $\phi_{\epsilon}(x)=\phi(\frac{|x|}{\epsilon}), \ \ \psi_R(x):=\psi(\frac{|x|}{R})$, $\phi$ and $\psi$ are smooth functions in $\R$ with the properties $0\leq \phi, \psi\leq 1$, with supports of $\phi, \psi$ in $(1,\infty)$ and $(-\infty, 2)$ respectively and $\phi(t)=1$ for $t\geq 2$ and $\psi=1$ for $t\leq 1$.

Let $u$ be any solution of Equation \eqref{eq:prob}, then $u$ is smooth away from the origin and hence $(x\cdot\na u)\phi_{\epsilon, R}\in C^2_c(\bar\Omega)$. Multiplying Equation \eqref{eq:prob} by $(x\cdot\na u)\phi_{\epsilon, R}$ and integrating by parts we obtain,
\begin{eqnarray}\label{4}
\Iom\na u\na\big((x\cdot\na u)\phi_{\epsilon, R}\big)dx-\int_{\bdw}\frac{\pa u}{\pa\nu}(x\cdot\na u)\phi_{\epsilon,R}dS &=&\mu\Iom|u|^{2^{*}-2}u(x\cdot\na u)\phi_{\epsilon, R}dx \no \\
&+&\Iom\frac{|u|^{2^{*}(t)-2}u}{|x|^t}(x\cdot\na u)\phi_{\epsilon, R}dx \no\\
&+&\Iom a(x)u(x\cdot\na u)\phi_{\epsilon, R}dx
\end{eqnarray}
RHS can be simplified as follows:
\begin{eqnarray}\label{5}
\text{RHS}&=& \frac{\mu}{2^{*}}\Iom\na(|u|^{2^{*}})\cdot x\phi_{\epsilon,R}dx+\frac{1}{2^{*}(t)}\Iom\na(|u|^{2^{*}(t)})\cdot x\frac{\phi_{\epsilon,R}}{|x|^t}dx\no\\
&+&\frac{1}{2}\Iom a(x)\na(|u|^2)\cdot x\phi_{\epsilon,R}dx\no\\
&=&-\mu\displaystyle\left(\frac{N-2}{2}\right)\Iom|u|^{2^{*}}\phi_{\epsilon,R}dx-\frac{\mu}{2^{*}}\Iom|u|^{2^{*}}(x\cdot\na\phi_{\epsilon,R})dx\no\\
&-&\frac{N-2}{2}\Iom\frac{|u|^{2^{*}(t)}}{|x|^t}\phi_{\epsilon,R}dx-\frac{1}{2^{*}(t)}\Iom\frac{|u|^{2^{*}(t)}}{|x|^t}(x\cdot\na\phi_{\epsilon,R})dx\no\\
&-&\frac{N}{2}\Iom a(x)|u|^2\phi_{\epsilon,R}dx-\frac{1}{2}\Iom a(x)|u|^2(x\cdot\na\phi_{\epsilon,R})dx\no\\
&-&\frac{1}{2}\Iom|u|^2(x\cdot\na a)\phi_{\epsilon,R}dx
\end{eqnarray}
As $|x\cdot\na\phi_{\epsilon,R}|=|x(\psi_R\na\phi_{\epsilon}+\phi_{\epsilon}\na\psi_R)|\leq C$, by using dominated convergence theorem, it is easy to check that 
\begin{eqnarray}\label{6}
\lim_{R\to\infty}\lim_{\epsilon\to 0}\text{RHS} &=&-\displaystyle\left(\frac{N-2}{2}\right)\left(\mu\Iom|u|^{2^{*}}dx+\Iom\frac{|u|^{2^{*}(t)}}{|x|^t}dx\right)\no\\
&-&\frac{N}{2}\Iom a(x)u^2dx-\frac{1}{2}\Iom|u|^2(x\cdot\na a)dx.
\end{eqnarray}
Following the calculation in \cite[Theorem 4.1]{BS}, LHS of \eqref{4} can be estimated as
\begin{eqnarray}\label{7}
\text{LHS}&=&-\displaystyle\left(\frac{N-2}{2}\right)\Iom|\na u|^2\phi_{\epsilon,R}dx-\frac{1}{2}\int_{\bdw}\left(\frac{\pa u}{\pa \nu}\right)^2(x\cdot\nu)\phi_{\epsilon,R}dS\no\\
&-&\frac{1}{2}\Iom|\na u|^2(x\cdot\na\phi_{\epsilon,R})dx+\Iom(x\cdot\na u)(\na u\cdot\na\phi_{\epsilon,R})dx
\end{eqnarray}
Here we used the fact that, $x\cdot\na u=x\cdot\nu\frac{\pa u}{\pa\nu}$ on $\bdw$, since $u=0$ on $\bdw$.
First three terms in the right hand side of \eqref{7} can be estimated as before. For the last term we can see that
\begin{eqnarray*}
&&\lim_{R\to\infty}\lim_{\epsilon\to 0}\displaystyle|\Iom(x\cdot\na u)(\na u\cdot\na\phi_{\epsilon,R})dx|\\
&=&\lim_{R\to\infty}\lim_{\epsilon\to 0}|\Iom (x\cdot\na u)\big(\psi_R(\na u\cdot\na\phi_{\epsilon})+\phi_{\epsilon}(\na u\cdot\na\psi_R)\big)dx|\\
&\leq&\lim_{\epsilon\to 0} C_1\int_{\epsilon\leq|x|\leq 2\epsilon}|\na u|^2 dx+\lim_{R\to\infty}C_2\int_{R\leq|x|\leq 2R}|\na u|^2 dx\\
&=&0
\end{eqnarray*}
Therefore from \eqref{7} we get,
\begin{equation}\label{8}
\lim_{R\to\infty}\lim_{\epsilon\to 0}LHS=-\displaystyle\left(\frac{N-2}{2}\right)\Iom|\na u|^2dx-\frac{1}{2}\int_{\bdw}\left(\frac{\pa u}{\pa \nu}\right)^2(x\cdot\nu)dS
\end{equation}
combining \eqref{6} and \eqref{8}, and using Equation \eqref{eq:prob}, we obtain
$$-\Iom\displaystyle\left(a(x)+\frac{1}{2}
x\cdot\na a\right)|u|^2 dx+\frac{1}{2}\int_{\bdw}\left(\frac{\pa u}{\pa \nu}\right)^2(x\cdot\nu)dS=0.$$
Since $\Omega$ is star shaped with respect to the origin, 2nd term in the LHS of the above expression is nonnegative and by the assumptions of this Theorem, the 1st term is also nonnegative. hence by the principle of unique continuation $u=0$ in $\Omega$. This completes the proof.
\hfill$\square$
\end{proof}

\vspace{2mm}

{\bf Acknowledgment}: This work is supported by INSPIRE research grant 
DST/INSPIRE 04/2013/000152.

\label{References}

\end{document}